\UseRawInputEncoding 
\documentclass[12pt]{article}

\usepackage{hyperref}
\hypersetup{nesting=true,debug=true,naturalnames=true}
\usepackage{listings}
\usepackage{setspace}
\usepackage[font=small,labelfont=bf]{caption}
\usepackage{graphicx,amssymb,upref}
\usepackage{tikz, latexsym, amsmath, amscd, amsfonts, array, lmodern, enumerate,amsthm}
\usepackage[english]{babel}
\usepackage{color}
\def\red{}

\usepackage[all]{xy}
\usepackage{subfig}
\usepackage{float}
\restylefloat{table}
\CompileMatrices

\PassOptionsToPackage{linesnumbered}{algorithm2e} \usepackage{algorithm2e}

\let\<\langle
\let\>\rangle

\let\uml\"
     
\title{A randomized greedy algorithm for piecewise linear motion planning}

\author{Carlos Ortiz\footnote{Supported by a Conacyt scholarship.}, Adriana Lara\footnote{Partially supported by Grant SIP20201381.}, Jes\'us Gonz\'alez\footnote{The third author thankfully acknowledge the computer resources, technical expertise and support provided by the Laboratorio Nacional de Supercómputo del Sureste de México, a CONACYT member of the network of national laboratories.}, and Ayse Borat}  

\date{}
 
\newcommand{\TC}{\mathrm{TC}}
\newcommand{\cat}{\mathrm{cat}}
\newcommand{\D}{\mathrm{D}}
\newcommand{\bsd}{\mathrm{Sd}}
\newcommand{\SC}{\mathrm{SC}}
\newcommand{\C}{\mathrm{SD}}
\newcommand{\card}{\mathrm{Card}}

\newtheorem{theo}{Theorem}[section]
\newtheorem{lemma}[theo]{Lemma}
\newtheorem{corollary}[theo]{Corollary}

\newtheorem{definition}[theo]{Definition}

\newtheorem{ejems}[theo]{Examples}
\newtheorem{remark}[theo]{Remark}

\begin{document}
\maketitle

\begin{abstract}
We describe and implement a randomized algorithm that inputs a polyhedron, thought of as the space of states of some automated guided vehicle $\mathcal{R}$, and outputs an explicit system of piecewise linear motion planners for $\mathcal{R}$. The algorithm is designed in such a way that the cardinality of the outputed system is probabilistically close (with parameters chosen by the user) to minimal possible. This yields the first automated solution for robust-to-noise robot motion planning in terms of simplicial complexity (SC) techniques, a discretization of Farber's topological complexity $\TC$. Besides its relevance toward technological applications, our work revels that, unlike other discrete approaches to TC, the SC model can recast Farber's invariant without having to introduce costly subdivisions. We develop and implement our algorithm by actually discretizing Mac\'ias-Virg\'os and Mosquera-Lois' notion of homotopic distance, thus encompassing computer estimations of other sectional category invariants as well, such as the Lusternik-Schnirelmann category of polyhedra. 
\end{abstract}



%%%% **** The text of the paper starts here **** %%%%

\section{Introduction}\label{seccionintroductoria}
The notion of topological complexity (TC), introduced by Michael Farber in~\cite{F2}, is a mathematical model measuring the continuity instabilities in the motion planning problem of robots. The theoretical aspects of Farber's idea have been extensively studied by algebraic topologists through nearly twenty years. As a result, the concept has found deep and fruitful connections in homotopy theory. Nonetheless, the more computational aspects of the TC-ideas have seen limited developments, while actual engineering-minded TC-applications seem to be inexistent. The purpose of this paper is to help mend such a situation by focusing on the more applied features of Farber's TC. We take advantage of computational topology techniques and, more importantly, by using randomized algorithms, we settle a cumputer-based solution to the motion planning problem for autonomous systems under Farber's model. We implement our algorithms with successful and, as explained below, unexpected results.

We use the discretization of Farber's TC developed in~\cite{MR3778506}. Here is the basic idea (details are reviewed and generalized in the next section). Assume that the space of states of a given autonomous system is given by (the topological realization of) an abstract simplicial complex $K$ where, for practical reasons, we assume $K$ to be finite. Taking a linear order on the vertices of $K$, we consider the ordered simplicial product $K\times K$. In such a context, a \emph{piecewise linear motion planner} defined on a subcomplex $J$ of $K\times K$ is a chain of simplicial maps $\varphi_0,\varphi_1,\ldots,\varphi_c\colon J\to K$ ($c\geq 1$) satisfying that (i) each $\varphi_i$ is contiguous to the subsequent map $\varphi_{i+1}$ ($0\leq i<c$), and that (ii) $\varphi_0$ ($\varphi_c$) is the restriction to $J$ of the projection $K\times K\to K$ onto the first (second) factor. The rationale behind such a definition is that, for any given pair $(a,b)$ of initial-final points in the topological realization $\| J \|$ of $J$, the sequence of points $$a,\; \|\varphi_0\|(a,b),\;\|\varphi_1\|(a,b),\,\ldots\,,\;\|\varphi_c\|(a,b),\;b$$ flags a piecewise linear path in $\| K\|$ from $a$ to $b$ depending continuously on $a$ and $b$. A motion planner defined on the full complex $K\times K$ exists only in the (topologically trivial) case that $\| K\|$ is contractible. It is then natural to define $\SC_{\text{strict}}(K)$ as one less than the minimal cardinality of systems of piecewise linear motion planners whose domains cover $K\times K$. This yields a discretized approximation of Farber's TC in the sense that $\TC(\|K\|)\leq\SC_{\text{strict}}(K)$.

A central result in~\cite{MR3778506} (coming from the Simplicial Approximation Theorem) is that, if we want to have a discretized invariant of $K$ that actually recovers (rather than just approximates) Farber's topological complexity of the realization $\| K\|$, then the above construction has to be done within a limiting process where $K\times K$ is allowed to be ``sufficiently'' subdivided, for instance, by taking $b$-iterated barycentric subdivisions $\bsd^b(K\times K)$. This yields the \emph{simplicial complexity} (SC) of~$K$, $$\SC(K)=\lim_{b\to\infty}\SC^{\hspace{.25mm}b}_{\text{strict}}(K),$$ where $\SC^{\hspace{.25mm}b}_{\text{strict}}(K)$ is defined in terms of $\bsd^b(K\times K)$ in the same way as $\SC_{\text{strict}}(K)$ is defined in terms of $K\times K$ (so $\SC^0_{\text{strict}}(K)=\SC_{\text{strict}}(K)$). Then, as shown in \cite[Theorem~3.5]{MR3778506},
\begin{equation}\label{mtr}
\SC^0_{\text{strict}}(K\times K)\geq \SC^1_{\text{strict}}(K\times K)\geq\SC^2_{\text{strict}}(K\times K)\geq\cdots\geq\SC(K)=\TC(\|K\|),
\end{equation}
recovering Farber's topological complexity of $\| K\|$. Our main theoretical result, Theorem~\ref{sc=tc} below, is a generalization of~(\ref{mtr}). This is achieved by introducing the notion of contiguity distance between simplicial maps (Definition~\ref{condis}), a discretization of the concept of homotopic distance studied recently in~\cite{HD}.

\begin{remark}\label{notosubdivide}{\em
The point above is that, by taking iterated barycentric subdivisions, we can recast Farber's invariant by means of a finite model. However, any person familiar with the computational aspects involving subdivisions will immediately realize that the proposal based on $b$-iterated barycentric subdivisions becomes computationally prohibitive as $b$ grows. Thus, in principle, the use of SC to evaluate TC carries a compromise between how close we want to approximate $\TC$ versus how long we are willing to wait for a computer to perform a calculation involving a large number of subdivisions. The following considerations explain that, in some cases, such compromise can be minimized without diminishing the effectiveness of the approach.
}\end{remark}

An initial way to soften the compromise above is to allow for coarser/partial subdivisions. The option is however still not satisfactory, as there is no indication \emph{a priori} of what portions of $K\times K$ (and how much) should be subdivided. With the aid of human geometric intuition, the coarser subdivision alternative was used in~\cite[Section~4]{MR3778506} to recast, after a single subdivision, the fact that the topological complexity of the circle (the natural benchmark in~\cite{MR3778506}) is 1, namely, 
\begin{equation}\label{goalinicial}
\SC^{\hspace{.2mm}b}_{\text{strict}}(\partial\Delta^2)=1\mbox{, \ for $b\geq1$.}
\end{equation}

With these preliminaries, we are ready to describe the main application-minded achievements of this paper. Firstly, we lay the grounds for a randomized algorithm that is able to perform, without human aid, evaluation tasks such as~(\ref{goalinicial}). Secondly, and more relevant for actual applications, we set a corresponding computer implementation which, besides of performing automated TC estimations, produces probabilistically close-to-optimal systems of piecewise linear motion planners. Optimality here means that, in each inputed case, the cardinality of the outputed system is close to been minimal possible (i.e., close to the TC value of the complex under consideration), where closeness is controlled by means of user-given parameters in the probabilistic algorithm. In brief, our implementation outputs robust-to-noise system of piecewise linear motion planners, with a reliability that depends on the given parameters.

As a byproduct, we find somehow surprisingly that the subdivision ingredient implicit in~(\ref{mtr}) can be unnecessary in a SC-recasting of Farber's TC, thus nullifying the critical compromise noted in Remark~\ref{notosubdivide}. This is a crucial feature for eventual technological TC-applications. For instance, running on a laptop computer, our algorithm finds in a bit less than a minute a system of piecewise linear motion planners giving in fact
\begin{equation}\label{sinsubdividir}
\SC^{\hspace{.2mm}0}_{\text{strict}}(\partial\Delta^2)=1.
\end{equation}
This should be compared to the fact that, starting with the barycentric subdivision $\bsd^1(\partial\Delta^2\times \partial\Delta^2)$, our computer implementation finds in about 10 hours running time a system of piecewise linear motion planners recovering the weaker fact~(\ref{goalinicial}). The reported running times (obtained on the same laptop computer) are given for comparison purposes: the computation is about 600 times (unnecessarily) more complex by allowing a single subdivision. Naturally, the use of powerful computational clusters results in significant time reductions, thus allowing to replace the circle $\partial\Delta^2$ by complexes arising from actual applications. The moral is that, in general, an SC-approach to motion planning applications tends to perform satisfactorily well by ignoring the subdivision component. Actual applications of the SC-ideas to concrete problems by using powerful computer clusters will be addressed in a future publication.

\begin{remark}{\em
The ability to recast Farber's TC in discrete terms without the need of subdivisions seems to be a convenient advantage of our SC-model over other discrete approaches. Specifically, the \emph{discrete topological complexity} introduced in~\cite{MR3834677} of a simplicial complex whose geometric realization is a circle turns out to be one unit larger than $\TC(S^1)$~\cite[Theorem~5.6]{MR3834677}. Likewise, (if no subdivisions are allowed) the value of Tanaka's \emph{combinatorial complexity} on the minimal finite space model of a circle is two units higher than $\TC(S^1)$~(\cite[Example~3.7]{MR3773738}).
}\end{remark}

The first part of the paper is devoted to reviewing the basic tools (Section~\ref{secback}) and developping the theoretical basis (Section~\ref{CD}) that show that our model recasts homotopy notions such as Farber's topological complexity and Lusternik-Schnirelmann's category. This is done via Mac\'ias-Virg\'os and Mosquera-Lois' homotopic distance, a concept whose discretization is achieved by means of Theorem~\ref{sc=tc}. The second part of the paper has an applied focus: Section~\ref{sda} describes the main processes in our algorithm, while Section~\ref{secndeexpmts} illustrates, via experimentation, the power that the technique has even when the subdivision ingredient is omitted.

\section{Background}\label{secback}
\subsection{Topological complexity via homotopic distance}
Let $P(X)$ stand for the free path space on a topological space $X$. Farber's topological complexity of~$X$, $\TC(X)$, is the sectional category of the end-point evaluation map $e\colon P(X)\to X\times X$, i.e., the fibration taking a free path $\gamma\in P(X)$ to the pair $(\gamma(0),\gamma(1))$. In other words, $\TC(X)+1$ is the smallest cardinality of open covers $\{U_i\}_i$ of $X\times X$ so that $e$ admits a continuous section $\sigma_i$ on each $U_i$. Note that we use reduced terms, so that a contractible space has zero topological complexity. The open sets $U_i$ in such an open cover are called {\it local domains}, the corresponding sections $\sigma_i$ are called {\it local motion planners}, and the family of pairs $\{(U_i,\sigma_i)\}$ is called a {\it system of local motion planners} for $X$. A system of local motion planners is said to be optimal if it has $\TC(X)+1$ local domains. In view of the continuity requirement on local rules, an optimal motion planner minimizes the possibility of accidents in the performance of a robot moving in a noisy environment.  

When $X$ is the realization of an abstract simplicial complex, the openness requirement on local domains can be replaced, without altering the numerical value of $\TC(X)$, by requiring that local domains are subdivided subcomplexes (see~\cite{MR3778506}, or the adaptation to the simplicial analogue of homotopic distance in the proof of Theorem~\ref{sc=tc} below). Furthermore, a local motion planner is really a (local) homotopy in disguise:

\begin{lemma}[{\cite[Lemma 4.21]{F}}]\label{motivation} The evaluation map $e\colon P(X)\to X\times X$ admits a section on a subset $A$ of $X\times X$ if and only if the restrictions to $A$ of the two projections to the axes $\pi_i\colon X\times X \to X$, $i=1,2$, are homotopic.
\end{lemma}

Consequently, $\TC(X)=\D(p_1,p_2)$, where: 

\begin{definition}[{\cite{HD}}]\label{homotopydistance}
The homotopic distance $\D(f,g)$ between two continuous maps $f,g\colon X\to Y$ is one less than the minimal cardinality of covers of $X$ by open sets $U$ on each of which the restricted maps $f_{|U}$ and $g_{|U}$ are homotopic. 
\end{definition}

\subsection{Contiguity vs.~homotopy}\label{shudyt}
For details on the facts reviewed in this subsection, the reader can consult standard references, such as~\cite[Chapter~3]{Spanier}.

\begin{definition}\label{denincontgd}
For a positive integer $c$, and a pair of simplicial maps $\varphi,\varphi': J\to K$,
\begin{enumerate}
\item\label{primeritita} $\varphi$ and $\varphi'$ are said to be 1-contiguous provided $\varphi(\sigma)\cup \varphi'(\sigma)$ is a simplex of $K$ for any facet (i.e., maximal simplex) $\sigma$ of $J$.
\item\label{segunditita} $\varphi$ and $\varphi'$ are said to be $c$-contiguous if there is a sequence of maps $\varphi_0, \varphi_1, \cdots, \varphi_c: J\to K$, with $\varphi_0=\varphi$ and $\varphi_c=\varphi'$, such that $\varphi_{i-1}$ and $\varphi_i$ are 1-contiguous for each $i\in\{1,2,\ldots,c\}$.
\end{enumerate}
\end{definition}

The sequence of maps $\varphi_j$ in Definition~\ref{denincontgd}.\ref{segunditita} is called a contiguity chain of length~$c$ between $\varphi$ and $\varphi'$. We write $\varphi\sim_c\varphi'$ to mean that there is such a sequence, and write $\varphi\sim \varphi'$ to mean $\varphi\sim_c \varphi'$ for some $c$. The latter is an equivalence relation in the set of simplicial maps $J\to K$. The corresponding equivalence classes are called contiguity classes. Note that contiguity classes are preserved under composition.

\begin{remark}\label{realizar}{\em
Contiguity classes refine homotopy classes in the sense that the topological realizations of a pair of contiguous simplicial maps are homotopic through a piecewise affine homotopy. More importantly, as soon as highly subdivided complexes are allowed, this leads to a one-to-one correspondence between contiguity and homotopy classes. To make the latter statement precise (in Theorem~\ref{bus} below), we need:
}\end{remark}

\begin{definition}
\emph{(a)} For a simplex $\sigma$ of $K$, the open realized simplex $\langle\sigma\rangle\subseteq\|K\|$ consists of the barycentric combinations $\sum_{v\in\sigma}t_vv$ with each $t_v$ positive. Note that the realization of $\sigma$, $\|\sigma\|\subseteq\|K\|$, is the closure of $\langle\sigma\rangle$. \emph{(b)} A simplicial approximation of a continuous map $f\colon \|J\|\to \|K\|$ is a simplicial map $\varphi\colon J\to K$ such that $\|\varphi\|(x)\in\|\sigma\|$ provided $f(x)\in\langle\sigma\rangle$.
\end{definition}

For instance, a simplicial map is the only approximation of its geometric realization. On the other hand, for a subdivision $K'$ of $K$, the standard homeomorphism $\|K'\|\stackrel{=}{\to}\|K\|$ is approximated by any simplicial map $\varphi\colon K'\to K$ with the property that, for every vertex $v'$ of $K'$, the barycentric coordinate of $v'\in\|K\|$ at $\varphi(v')$ is positive. Further, simplicial composition of approximations approximates the corresponding topological composition.

\begin{theo}\label{bus} Uniqueness of approximations: Any two approximations of a continuous map are $1$-contiguous. Further, for finite $J$:
\begin{itemize}
\item[(i)] Existence of approximations: Any continuous map $f\colon \| J \|\to \|K\|$ admits an approximation $\varphi_b\colon\bsd^b(J)\to K$ for any $b$ large enough. In particular, if $\iota\colon \bsd^{b+1}(J)\to\bsd^b(J)$ is any approximation of the identity on $\|J\|$, then $\varphi_b\hspace{.3mm} \iota$ and $\varphi_{b+1}$ are 1-contiguous.
\item[(ii)] Contiguity vs.~homotopy: Recall that simplicial maps in the same contiguity class have homotopic topological realizations (Remark~\ref{realizar}). Conversly, given homotopic maps $f_0,f_1\colon \|J\|\to\|K\|$, there is a non-negative integer $b_0$ such that, for each $b\geq b_0$, any pair of approximations $\varphi_0,\varphi_1\colon\bsd^b(J)\to K$ of $f_0,f_1$, respectively, satisfy $\varphi_0\sim\varphi_1$.
\end{itemize}
\end{theo}

All complexes we deal with will be assumed to be finite. Additionally whenever we need to talk about products, they will be taken in the category of ordered complexes (fixing some linear order on the corresponding sets of vertices). The key point is that the topological realization of an ordered product of complexes is homeomorphic to the product of the topological realizations of the factors (see~\cite{ES52}). We do not require, though, that maps of ordered complexes preserve the given orderings.

\section{Contiguity distance}\label{CD}
We now adapt the viewpoint in~\cite{MR3778506} in order to recast, in simplicial terms, the homotopic distance between continuous functions. Throughout this section $\varphi,\varphi'\colon L\to K$ denote a fixed pair of simplicial maps. A subcomplex $J$ of~$L$ is said to be a $c$-contiguity subcomplex for $\varphi$ and $\varphi'$ if the restrictions $\varphi_{|J}$ and $\varphi'_{|J}$ are $c$-contiguous. We say that $J$ is a contiguity subcomplex for $\varphi$ and $\varphi'$ if it is a $c$-contiguity subcomplex for some $c$. The strict $c$-contiguity distance of $\varphi$ and~$\varphi'$, denoted by $\C_c(\varphi,\varphi')$, is defined as one less than the smallest cardinality of finite covers of $L$ by $c$-contiguity subcomplexes for $\varphi$ and $\varphi'$. The \emph{strict contiguity distance} of $\varphi$ and~$\varphi'$ is $\C(\varphi,\varphi'):=\lim_c\C_c(\varphi,\varphi')$, the eventual constant value of the monotonic sequence
$$
\C_1(\varphi,\varphi')\geq \C_2(\varphi,\varphi') \geq \C_3(\varphi,\varphi') \geq \cdots. 
$$
\red{The concept of strict contiguity distance was introduced in~\cite[Definition 8.1]{HD} (see also~\cite{SDboratetall}) under the name of contiguity distance. We use the shorter name for the actual notion recovering homotopic distance (see Definition~\ref{condis} and Theorem~\ref{sc=tc} below).}

\begin{ejems}\label{exacon}{\em
Note that the subcomplex generated by any simplex of $L$ is a contiguity subcomplex for $\varphi$ and $\varphi'$ provided (the topological realization of) $K$ is connected. Connectedness of complexes is an assumption that will be in force from this point on. The relation $\varphi\sim_c\varphi'$ can be expressed through the equality $\C_c(\varphi,\varphi')=0$, while the equality
$\mbox{$\SC_{\text{strict}}(K)=\C(\pi_1,\pi_2)$}$ holds by definition, where $\pi_1,\pi_2\colon K\times K\to K$ are the two simplicial projections.
}\end{ejems} 

Remark~\ref{realizar} readily gives
\begin{equation}\label{readilygives}
\D(\|\varphi\|,\|\varphi'\|)\leq \C(\varphi,\varphi')
\end{equation}
We next use the full power of Theorem~\ref{bus} in order to adjust the definition of SD and replace the inequality in~(\ref{readilygives}) by an equality. For each positive integer $b$, fix a simplicial approximation $\iota=\iota_b\colon \bsd^b(L)\to \bsd^{b-1}(L)$ of the identity on $\|L\|$. We abuse notation and write $\iota\colon\bsd^b(L)\to \bsd^{b'}(L)$ for the obvious iterated composition of $\iota$ maps ($b\geq b'$). Set $\C_c^b(\varphi,\varphi'):=\C_c(\varphi_b,\varphi'_b)$ and $\C^b(\varphi,\varphi'):=\C(\varphi_b,\varphi'_b)$, the corresponding $c$-stabilized value, where the maps $\varphi_b, \varphi'_b\colon \bsd^b(L)\to K$ stand for $\varphi\circ\iota$ and $\varphi'\circ\iota$, respectively. Note that the contiguity class of $\varphi_b$ (respectively~$\varphi_b'$) depends only on that of $\varphi$ (respectively~$\varphi'$). Although the numbers $\C^b_c(\varphi,\varphi')$ depend in principle on the chosen set of approximations $\iota$, the proof of~\cite[Lemma~3.3]{MR3778506} applies word for word to show that the $c$-stabilized value $\C^b(\varphi,\varphi')$ is independent of the chosen approximations. Futher, if $J$ is a subcomplex of $\bsd^b(L)$ and $\lambda\colon\bsd(J)\to J$ is an approximation of the identity on $\|J\|$, then the two compositions in the diagram 
$$\xymatrix{
J\;\ar@{^{(}->}[r] & \bsd^b(L) \\
\bsd(J)\;\ar@{^{(}->}[r] \ar[u]^\lambda& \bsd^{b+1}(L)\ar[u]_\iota
}$$
are 1-contiguous, as they are approximations of the inclusion $\|J\|\subseteq\|L\|$. This shows $\C^b_c(\varphi,\varphi')\geq\C^{b+1}_{c+2}(\varphi,\varphi')$, which yields the monotonic sequence
\begin{equation}\label{siempresimono}
\C^0(\varphi,\varphi')\geq\C^1(\varphi,\varphi')\geq\C^2(\varphi,\varphi').\geq\cdots\geq0.
\end{equation}

\begin{definition}\label{condis} The contiguity distance between $\varphi$ and $\varphi'$ is $$\D(\varphi,\varphi'):=\lim_b\C^b(\varphi,\varphi'),$$ the $b$-stabilized value of the monotonic sequence~(\ref{siempresimono}).
\end{definition}

Since $\D(\varphi,\varphi')=\C^b(\varphi,\varphi')$ holds for any large enough index $b$, (\ref{readilygives}) becomes
\begin{equation}\label{bascom}
\D(\|\varphi\|,\|\varphi'\|)\leq \D(\varphi,\varphi').
\end{equation}
We now adapt the argument in the proof of~\cite[Theorem~3.5]{MR3778506} to prove:

\begin{theo}\label{sc=tc}
Equality holds in~(\ref{bascom}).
\end{theo}

\begin{proof}
Let $\D(\|\varphi\|,\|\varphi'\|)=k$ and choose an open covering $\{U_0,\ldots,U_k\}$ of $\|L\|$ as in Definition~\ref{homotopydistance}. Since $L$ is finite, there is an integer $b\geq0$ such that the realization of each simplex of $\bsd^b(L)$ is contained in some $U_j$ ($0\leq j\leq k$). For each $j\in\{0,1,\ldots,k\}$, let $L_j$ be the subcomplex of $\bsd^b(L)$ generated by those simplices whose realization is contained in $U_j$. Then $L_0,L_1,\ldots,L_k$ cover $L$, and $\|\varphi\|_{|L_j}$ and $\|\varphi'\|_{|L_j}$ are homotopic for $0\leq j\leq k$. By Theorem~\ref{bus}, we can pick large enough integers $b'$ and $c$ such that, for each $j\in\{0,1,\ldots,k\}$, the two composites 
$$\xymatrix{
\bsd^{b+b'}(L_i)\, \ar@{^{(}->}[r] & \bsd^{b+b'}(L) \ar@<.5ex>[r]^>>>>>{\varphi_{b+b'}}
\ar@<-.5ex>[r]_>>>>>{\varphi'_{b+b'}} & \,K
}$$
are $c$-contiguous. This yields $\D(\|\varphi\|,\|\varphi'\|)=k\geq\C^{b+b'}_c(\varphi,\varphi')\geq\D(\varphi,\varphi')$.
\end{proof}

\red{Since a continuous map $f\colon\|L\|\to\|K\|$ is homotopic to the geometric realization of any of its simplicial approximations, Theorem~\ref{sc=tc} and~\cite[Proposition~2.2]{HD} immediately yield:}

\begin{corollary}
\red{For a pair of contiguous maps $f,f'\colon\|L\|\to\|K\|$, we have $$\D(f,f')=\D(\varphi,\varphi'),$$ where $\varphi,\varphi'\colon\bsd^b(L)\to K$ are any corresponding simplicial approximations (for some large enough $b$).}
\end{corollary}

\red{Actually, Theorem~\ref{sc=tc} allows us to import to the simplicial realm properties of homotopic distance proved in~\cite{HD}. For instance:}

\begin{corollary}
$\D(\varphi,\varphi')$ is a contiguity invariant of the pair of maps $(\varphi,\varphi')$. \red{Explicitly, if the simplicial maps in the diagram
$$
\xymatrix{
L \ar@<.5ex>[r]^{\varphi} \ar@<-.5ex>[r]_{\varphi'} & K \ar[d]^\beta\\
L' \ar@<.5ex>[r]^{\psi} \ar@<-.5ex>[r]_{\psi'} \ar[u]^\alpha& K'
}
$$
satisfy $\beta\circ\varphi\circ\alpha\sim\psi$ and $\beta\circ\varphi'\circ\alpha\sim\psi'$, with $\alpha$ and $\beta$ contiguity equivalences (i.e., they have inverses up to contiguity), then $\D(\varphi,\varphi')=\D(\psi,\psi')$}
\end{corollary}

\begin{corollary}
The inequality $\D(\varphi,\varphi')\leq\TC(\|K\|)$ holds, as well as the triangular inequality $\D(\varphi,\varphi'')\leq\D(\varphi,\varphi')+\D(\varphi',\varphi'')$ and the ``sublogarithmic'' inequalities
\begin{align*}
\D(\,\varphi\circ\psi,\varphi'\circ\psi'\,)&\leq\D(\varphi,\varphi')+\D(\psi,\psi')\\\D(\varphi\times\phi,\varphi'\times\phi')&\leq\D(\varphi,\varphi')+\D(\phi,\phi')
\end{align*}
for $\psi,\psi'\colon M\to L$, $\phi,\phi'\colon L'\to K'$ and $\varphi,\varphi',\varphi''\colon L\to K$.
\end{corollary}

\begin{corollary}
$\cat(\|K\|)=\D(\ast,1_K)=\D(\iota_1,\iota_2)$ and $\TC(\|K\|)=\D(\pi_1,\pi_2)$. Here $1_K$ is the identity map on $K$, $\pi_i$ are the simplicial projections in Examples~\ref{exacon}, and $\iota_i\colon K\to K\times K$ are the axial inclusions $\iota_1(v)=(v,v_0)$ and $\iota_2(v)=(v_0,v)$.
\end{corollary}

\section{Algorithms}\label{sda}
We now describe a randomized algorithm whose implementation yields reasonable estimations of the contiguity distance between simplicial maps. As described in the introductory section, in the case of simplicial complexity, our implementation constructs reasonably close-to-optimal systems of piecewise linear local motion planners.

We represent an (abstract  simplicial) complex $K$ by a pair $(V_K,F_K)$, where $V_K$ is the list of the vertices of $K$, and $F_K$ is the list of the facets (i.e.~maximal simplices) of~$K$. This encoding is convenient because of Definition~\ref{denincontgd}.\ref{primeritita} and, more importantly for our purposes, since a covering of a given complex is completely determined by a covering of its facets (see the discussion at the beginning of Subsection~\ref{seccontsbcxs}).

In what follows we assume given an algorithm, {\it Contiguous,} that checks whether two simplicial maps are contiguous\footnote{Note that if a simplicial map $\varphi$ is contiguous (in the sense of Definition~\ref{denincontgd}.\ref{primeritita}) to a vertex map $\varphi'$, then $\varphi'$ must be simplicial. In this paper we assume that {\it Contiguous} has been implemented solely on the basis of Definition~\ref{denincontgd}.\ref{primeritita}, so to admit any two vertex maps as input, necessarily reporting {\bf false} whenever one of the input parameters is not simplicial. In fact, under these assumption, {\it Contiguous\hspace{.3mm}}$(\varphi,\varphi)$ can be used to check if a vertex map $\varphi$ is simplicial.}. Such a function is available in standard mathematical software systems such as SageMath. In fact, our algorithms are easily implementable in SageMath, a system which has many other convenient features for our purposes ---such as a function for spelling out (as a list of facets) the ordered simplicial product of ordered complexes.

\subsection{Randomized local search of contiguity chains}
Recall that, for finite abstract simplicial complexes $J$ and $K$, Theorem~\ref{bus} characterizes homotopy classes of continuous maps $\|J\|\to\|K\|$ in terms of contiguity classes of maps $J\to K$ through a limiting process that takes finer and finer subdivisions of $J$. This classical fact is generalized in~\cite{MR3124944} by  introducing a certain ``contiguity complex'' $\text{Contig}(J,K)$ that approximates, as $J$ becomes sufficiently subdivided, the homotopy type of the function space of continuous maps $\|J\|\to\|K\|$. What is relevant for us is to remark that Blumberg and Mandell propose a certain randomized algorithm in order to study the rate of growth (under subdivision of the domain) of the components of $\text{Contig}(J,K)$. We adapt Blumberg-Mandell's algorithm (without the subdivision component) for our computer implemented search of piecewise linear local motion planners.

The randomized algorithm {\it LocalSearch} in this section inputs  a pair of simplicial maps $\varphi,\varphi'\colon J\to K$, a positive integer $M$, and a probability parameter $r$ ($r\in[0,1]$), and outputs a list $\Psi$ that either contains a chain of maps
\begin{equation}\label{cadena}
\varphi_0, \varphi_1, \cdots, \varphi_c: J\to K,\quad c\leq M,
\end{equation}
satisfying the conditions in Definition~\ref{denincontgd}.\ref{segunditita} or, else, is empty, in case the randomized search for~(\ref{cadena}) is not successful. Starting at $\varphi$, the search is done through a random walk with local steps in the space of simplicial maps $J\to K$. Actual steps in the walk are tried (and recorded in $\Psi$), at most $M$ times on a ``greedy'' basis, with the probability parameter~$r$ used to break a situation where no greedy step has been taken after a number of consecutive tries. A greedy step is actually taken whenever the distance from the goal map $\varphi'$ to the current position of the walk is larger than the distance from $\varphi'$ to the potentially new position of the walk. Here, the distance between simplicial maps $f,g\colon J\to K$ is defined as
$$
d(f,g)=\sum d_K(f(v),g(v)),
$$
where the summation runs over the vertices $v$ of $J$, and $d_K$ stands for the graph distance on the 1-skeleton of $K$.

\bigskip\begin{algorithm}[H]
\If{$\text{\it Contiguous\hspace{.5mm}}(\varphi,\varphi')$}{
\Return $\Psi\gets\{\varphi,\varphi'\}$}

$\Psi\gets\{\varphi\} \hspace{.2mm}$; \ $\phi\gets\varphi$ 

\For{\label{linea1}$i\gets1$ \KwTo $M$}{

$f\gets\phi\, ; \;\; v_j\gets \text{\it RandomVertex\hspace{.5mm}}(V_{J})$\label{line2}

$f(v_j)\gets \text{\it RandomVertex\hspace{0mm}}\left(V_{K}\setminus\{\phi(v_j)\}\rule{0mm}{4mm}\right)$\label{line3}

$p\gets \text{\it RandomNumber\hspace{.5mm}}(0,1)$

\If{\label{line5}\text{Contiguous\hspace{.5mm}}$(\phi,f)$ { \bf and } $\left(\rule{0mm}{4mm}\hspace{.4mm}p<r \text{\bf \hspace{.7mm} or\hspace{2mm}} d(f,\varphi')<d(\phi,\varphi')\right)$} {
$\Psi\gets\Psi\cup\{f\}$; \label{line6} $\hspace{.3mm}\phi\gets f$\label{line4}

\If{$\phi=\varphi'$} {\Return $\Psi$}
}}
\Return $\Psi\gets\varnothing$

\vspace{2mm}
\caption{{\it LocalSearch.} Current element in $\Psi$ is recorded by $\phi$.}
\end{algorithm}

\bigskip
We use the following variant of \emph{LocalSearch} in our actual implementation. Start by noticing that, at each step of the iterative process (line~\ref{linea1}), the potential new position~$f$ in the walk ---a random single-vertex variant of $\phi$--- may fail to be contiguous to the current position $\phi$, thus preventing us from taking the greedy step (line~\ref{line6}). The construction of $f$ (lines~\ref{line2} and~\ref{line3}) can then be replaced by a process that, first constructs a list $L_{\phi,w}$ of all the simplicial maps that differ from $\phi$ on a single randomly chosen vertex $w\in V_J$, and then chooses $f$ randomly from $L_{\phi,w}$ (provided $L_{\phi,w}\neq\varnothing$). Of course, with this modification, the first test condition in line~\ref{line5} can safely be removed. 

\subsection{Reduction of contiguity chains}
The random walk performed by {\it LocalSearch} has a local-step basis: any two consecutive maps $\varphi_{i}$ and $\varphi_{i+1}$ in a sequence~(\ref{cadena}) produced by {\it LocalSearch} differ only by their values at a single vertex of~$J$. In particular, it is usual that {\it LocalSearch} outputs a long sequence of (hundreds and even thousands of) consecutive contiguous maps in between $\varphi$ and $\varphi'$. This could be remedied by running a shortest path algorithm/heuristic on the contiguity graph generated by the maps $\varphi_i$ in the sequence. However, the construction of the graph is costly for large sequences. The simpler alternative described in this subsection works well given the random nature of \emph{LocalSearch}. 

\begin{algorithm}[ht!]
$\phi'\gets\{\varphi_0\}$; \ $j\gets 0$

\While{$j\neq c$}{

$i\gets c$

\While{not \text{Contiguous\hspace{.5mm}}$(\varphi_j,\varphi_i)$}{

$i\gets i-1$}

$\phi'\gets\phi'\cup\{\varphi_{i}\}$; \ $j\gets i$

}
\Return $\phi'$

\vspace{2mm}
\caption{{\it Reduce.}}
\end{algorithm}

The algorithm {\it Reduce} attempts to reduce the size of a contiguity sequence $\phi=(\varphi_0,\ldots,\varphi_c)$ by reporting the sequence $\phi'$ obtained from $\phi$ by discarding the terms $\varphi_{\ell+1}, \varphi_{\ell+2},\ldots,\varphi_{\ell+m-1}$ in chunks $\varphi_{\ell}, \varphi_{\ell+1},\ldots,\varphi_{\ell+m}$ whenever $m>1$ is maximal with $\varphi_{\ell}$ and $\varphi_{\ell+m}$ contiguous.

\subsection{Randomized contiguity subcomplexes}\label{seccontsbcxs}
For simplicial maps $\psi,\psi'\colon L\to K$, the task of finding coverings of $L$ by $c$-contiguity subcomplexes can be focused on facets of $L$. Indeed, by restricting to facets, a covering $\mathcal{C}$ of $L$ by $c$-contiguity subcomplexes yields (in a non-unique way) a partition $\mathcal{P}$ of the facets of $L$. In such a situation, if $J_P$ stands for the subcomplex of~$L$ generated by the facets in a given $P\in\mathcal{P}$, then $\mathcal{C}_\mathcal{P}:=\{J_P\colon P\in\mathcal{P}\}$ is a covering of $L$ by $c$-contiguity subcomplexes with $\card(\mathcal{C}_\mathcal{P})\leq\card(\mathcal{C})$. The definition of $\C_c(\psi,\psi')$ can therefore be reformulated by limiting attention to coverings $\mathcal{C}_\mathcal{P}$ coming from a partition~$\mathcal{P}$ of the facets of $L$ as above. This is the viewpoint in the algorithms described next.

A random subset of the facets of $L$ will most likely fail to generate a contiguity subcomplex. Likewise, a random partition $\mathcal{P}$ of the facets of $L$ will most likely fail to produce a covering $\mathcal{C}_\mathcal{P}$ by contiguity complexes. A more careful randomized search is needed. As a first step, the randomized algorithm \emph{RCC} (\emph{RandomContiguitysubComplex}) in this subsection aims at constructing a maximal contiguity subcomplex for a given pair of simplicial maps. Then, in Subsection~\ref{seccioncoverings} we describe a randomized algorithm that searches for partitions $\mathcal{P}$ of the set of facets of $L$ that yield a covering $\mathcal{C}_\mathcal{P}$ by contiguity subcomplexes. Lastly, the size of such a covering $\mathcal{C}_\mathcal{P}$ is optimized by the algorithm in Subsection~\ref{seccioncoveringsreloaded}. The latter algorithm is crucial for our purposes, as it gives us a real chance to get at systems of piecewise linear local motion planners that are (near-to) optimal in the sense of Section~\ref{shudyt}. 

\begin{algorithm}[ht!]
	$\mathcal{O} \gets \left\lbrace \sigma \in F_L \colon \sigma \notin F_J \right\rbrace $\\

	\While{$\mathcal{O} \neq \varnothing $ }{
		$\sigma\gets \text{{\it RandomFacet\hspace{0.5mm}}}(\mathcal{O}) $\\
		$J'\gets \text{{\it SimplicialComplex\hspace{0.5mm}}}(J,\sigma)$\\			       $\phi\gets \text{{\it LocalSearch\hspace{0.5mm}}}(\psi\hspace{.2mm}|_{J'},\psi'\hspace{.2mm}|_{J'})$ \\
		\If{$ \phi \neq \varnothing $}{\Return $J'$}
		$\mathcal{O} \gets \mathcal{O} \setminus \left\lbrace \sigma \right\rbrace $	
		
	}	
	
	\Return $J$
	\vspace{3mm}
	
	\caption{\textit{AddFacet.} A random $\sigma$ ensures variety of results from multiple runs.}
\end{algorithm}

We start with the algorithm {\it AddFacet}, whose input is a pair of simplicial maps $\psi,\psi'\colon L\to K$, and a contiguity subcomplex $J$ for $\psi$ and $\psi'$. It is implicitly assumed that $J$ is generated by a set of facets of $L$. With this information, the algorithm looks randomly for the first facet $\sigma$ of $L$ not in $J$ that, together with $J$, generates a contiguity subcomplex $J'$. If such a facet $\sigma$ is found, {\it AddFacet} outputs~$J'$, otherwise $J$ is reported. Each testing is done by the algorithm {\it LocalSearch} with the restricted simplicial maps $\psi|_{J'}$ and $\psi'|_{J'}$ as parameters. Note that {\it AddFacet} requires in addition the two parameters $M$ and $r$ needed by {\it LocalSearch}.

The main algorithm in this subsection, {\it RCC}, is an iteration of {\it AddFacet}. Starting with two simplicial maps $\psi,\psi'\colon L\to K$ as input, \emph{RCC} applies \emph{AddFacet} recursively, using the output of the previous application as (part of) the input for the next application ($\psi$ and $\psi'$ are kept as the rest of the input for all iterations of {\it AddFacet}). The iteration starts by using the contiguity subcomplex generated by a randomly chosen facet of~$L$, which is a contiguity subcomplex in view of Examples~\ref{exacon}. The iteration is applied at most $\card(F_L)$ times, and is set to stop whenever the current application of {\it AddFacet} is unable to add an additional facet (this will hold, for instance, if the total complex $L$ has been identified as a contiguity complex for $\psi$ and $\psi'$). 

A slight generalization of \emph{RCC} will also be needed in what follows. The algorithm \emph{AddFacets} is an iteration of \emph{AddFacet} on the same grounds as \emph{RCC}, except that the starting contiguity subcomplex is a prescribed parameter.

For latter use in the global algorithm, an actual implementation of both \emph{AddFacets} and \emph{RCC} should keep track of the (reduced version of the) last non-empty  contiguity sequence~$\phi$ constructed by \emph{AddFacet}, for this provides us with an explicit contiguity chain for the restrictions of $\psi$ and $\psi'$ to the output of \emph{AddFacets} or \emph{RCC}.

\subsection{Randomized coverings by contiguity subcomplexes}\label{seccioncoverings}
For simplicial maps $\psi,\psi'\colon L\to K$, the algorithm {\it Covering} in this subsection constructs a partition $\mathcal{P}$ of $F_L$ as the one described in the initial paragraph of Subsetion~\ref{seccontsbcxs}. The process is an iteration of {\it RCC}. Assume we have constructed a family~$\mathcal{P}$ of pairwise disjoint subsets of $F_L$ such that each $P\in\mathcal{P}$ generates a contiguity subcomplex $J_P$ for $\psi$ and $\psi'$ ($\mathcal{P}$ is empty at the start of the process). Then we execute {\it RCC} with the restricted maps $\psi_{|I},\psi'_{|I}\colon I\to K$ as parameters, where $I$ is the subcomplex of $L$ generated by the facets of $L$ that do not lie in any $P\in\mathcal{P}$. The set of relevant facets of the resulting contiguity complex is appended to $\mathcal{P}$. The process is iterated until the resulting $\mathcal{P}$ partitions $F_L$. 

\begin{algorithm}[ht!]
	$A \gets F_L\hspace{.2mm}$

	$\,\mathcal{P} \gets \varnothing$\\
	
	\While{$A \neq \varnothing$}{
	       $I\gets \text{{\it SimplicialComplex\hspace{0.5mm}}}(A)$ \\
		$J \gets \text{{\it RCC\hspace{0.5mm}}}(\psi|_I,\psi'|_I)$\\
		$P\gets \left(Facets(J)\cap F_L\right) \setminus \left(\cup_{Q\in\mathcal{P}} Q\right)$ \\
		$\mathcal{P} \gets \mathcal{P} \cup \left\lbrace P \right\rbrace \hspace{.2mm}$; \,$A \gets A \setminus P$	}
	
	\Return $\mathcal{P}$
	
	\vspace{3mm}
	
	\caption{\textit{Covering.} }
\end{algorithm}

\subsection{Optimization of coverings}\label{seccioncoveringsreloaded}
When applied to the two simplicial projections $\pi_1,\pi_2\colon K\times K\to K$, the algorithm {\it Covering} in the previous subsection constructs systems of piecewise linear local motion planners in rather short time. In addition, in cases where we know the topological optimal $\TC(\|K\|)+1$, some sporadic runs of \emph{Covering} report systems with cardinality reasonably close to optimal. In this subsection we describe the randomized algorithm \emph{OptimizedCovering} that addresses all other situations, i.e., those where the systems reported by \emph{Covering} appear to have too many domains. This is achieved by using a greedy strategy that attempts to reduce the number of piecewise linear local domains by increasing the size of large domains. \emph{OptimizedCovering} is the key process giving the convenient performance of our implementation without the need of expensive subdivisions, as advertised in the introductory section. 

\begin{algorithm}[ht!]
	
	$\mathcal{P} \gets \text{{\it Covering\hspace{0.5mm}}}(\psi,\psi');\;\;$
	$i \gets 0;\;\;$ $j\gets 0$ \\	
	\While{$i<N \text{ and \text{{\it Card\hspace{0.5mm}}}} (\mathcal{P})>t   $}{
	       $i\gets i+1;\;\;$ $\mathcal{P}' \gets \mathcal{P};\;\;$\\
		$\mathcal{P} \gets \text{{\it Order\hspace{0.5mm}}} (\mathcal{P}) $\\
		$D\gets \bigcup\limits_{k=j}^{\text{{\it Card\hspace{0.5mm}}}(\mathcal{P})-1} P_{k}$ \\
	
		$J_P \gets \text{{\it RCC\hspace{0.5mm}}}(\psi|_{J_D},\psi'|_{J_D})$\\
		
		$P\gets \text{{\it LargestSet\hspace{0.5mm}}}(P,P_j)$ \\
	
	       $D\gets \bigcup\limits_{k=\max(0,j-1)}^{\text{{\it Card\hspace{0.5mm}}}(\mathcal{P})-1} P_{k}$\\
	       
	       $J_Q \gets \text{{\it AddFacets\hspace{0.5mm}}}(J_P,\psi|_{J_D},\psi'|_{J_D})$\\
	       
		\If{$j>0$}{$Q\gets \text{{\it LargestSet\hspace{0.5mm}}}(Q,P_{j-1})$ \\
		}
	
		$\mathcal{P} \gets \text{{\it DeleteVoids\hspace{0.5mm}}}(\{P_0,P_1,\ldots,P_{j-2},Q,P_{j-1}-Q,P_j-Q,\ldots,P_p-Q\})$ 
		
		\If{$\text{{\it Card\hspace{0.5mm}}}(\mathcal{P}')<\text{{\it Card\hspace{0.5mm}}}(\mathcal{P})$}
		{$\mathcal{P}\gets \mathcal{P}' $}
		$j\gets j+1$\\
		\If{$j>\text{{\it Card\hspace{0.5mm}}}(\mathcal{P})-1$}{$j\gets0$}		
	}		
	\Return $\mathcal{P}$	
	\vspace{3mm}
	\caption{\textit{OptimizedCovering.}}
\end{algorithm}

{\it OptimizedCovering} starts with a partition $\mathcal{P}$ of $F_L$ produced by {\it Covering}, and goes into an iterative process that aims at shortening the length of $\mathcal{P}$. Explicitly, assume that, after the $i$-th stage of the iteration, the original partition has evolved to become the partition $\{P_0,P_1,\ldots,P_p\}$.
Then, using a control variable $j$ (inductively assumed to lie in between $0$ and $p$), the $(i+1)$-st recursive stage of {\it OptimizedCovering} performs the following actions:
\begin{enumerate}[(1)]
\item Order the partition so that $\card(P_0)\geq\card(P_1)\geq\cdots\geq\card(P_p)$.
\item Use the algorithm {\it RCC\hspace{.7mm}} to generate a random contiguity subcomplex~$J_P$ of $J_{P_j\cup P_{j+1}\cup\cdots \cup P_p}$. Keep in $P$ the largest set of $P$ and $P_j$.
\item Use the algorithm {\it AddFacets} to add as many facets of
\begin{equation}\label{tejiendo}
P_{j-1}\cup P_j\cup\cdots \cup P_p
\end{equation}
as possible to $P$, so to produce a contiguity subcomplex $J_{Q}$ containing $J_P$.  If $j=0$, part $P_{j-1}$ is inexistent in~(\ref{tejiendo}). If $j>0$, keep in $Q$ the largest set between $Q$ and $P_{j-1}$.
\end{enumerate}
The $(i+1)$-st iteration of the process then finishes by constructing the new (optimized) partition
\begin{equation}\label{removevoids}
\{P_0,P_1,\ldots,P_{j-2},Q,P_{j-1}-Q,P_j-Q,\ldots,P_p-Q\}\rule{0mm}{4mm},
\end{equation}
where empty parts are eliminated. If $j=0$, parts $P_0$, $P_1,\ldots,P_{j-2}$ and $P_{j-1}-Q$ are inexistent in~(\ref{removevoids}). Lastly, in preparation for the next iteration, the control variable~$j$ is incremented by one, unless its value has to be reset to zero so to meet the inductive hypothesis on $j$. The rationale behind this process is to use the best possible contiguity subcomplex that can be built from the facets in $\cup_{i\geq j}P_i$ (step (2)) to greedily improve on the cardinality of $P_{j-1}$ (step (3)). The net effect of such a recursive process is that new longer portions start ``bubbling up'' in the most recently produced partitions, while shorter portions tend to disappear from previously constructed partitions, as their elements get added to the longer emerging parts. As a consequence, the new partitions tend to have fewer domains than the old partitions.

For better results, the main recursive loop in the process above is meant to be repeated a large number of times (indicated by a parameter $N$ prescribed by the user). The pseudocode we describe uses in addition a parameter $t$ (also determined by the user) that breaks the recursion as soon as a partition with $t$ elements or less is achieved. The value of $t$ is to be provided on the basis of getting a ``short enough'' partition, either because the user would be happy with the prescribed bound, or simply because it does not make sense to insist on getting a partition of length smaller that the optimal $\TC(\|K\|)+1$ (if the latter number is known in advance, say by theoretical but non-constructive means).

\section{Computational results}\label{secndeexpmts}
\subsection{The circle}
Let $K:=\partial\Delta^2$, i.e., the most efficient triangulation of the circle, with vertices labelled $0,1,2$. The (realization of the ordered) product structure on $K\times K$ is depicted in Figure~\ref{s1xs1}, where opposite sides of the external square are identified as indicated. Using parameters $M=1000$ and $r=0.1$ for \emph{LocalSearch}, \emph{OptimizedCovering} yields, in about 57 seconds, the contiguity covering $\{J_0,J_1\}$ of $K\times K$ for the two projections, where $J_i$ is generated by the $i$-labelled triangles in Figure~\ref{s1xs1}. 

\begin{figure}[H]
$$\begin{tikzpicture}[x=.6cm,y=.6cm]
\draw(0,0)--(6,0); \draw(0,0)--(0,6);
\draw(0,6)--(6,6); \draw(6,0)--(6,6);
\draw(0,0)--(6,6); \draw(6,2)--(2,6); 
\draw(0,2)--(6,2); \draw(0,4)--(6,4); 
\draw(2,0)--(2,6); \draw(4,0)--(4,6);
\draw(2,0)--(4,2); \draw(4,2)--(6,0);
\draw(0,2)--(2,4); \draw(2,4)--(0,6);
\node [below] at (0,0) {\scriptsize$0'$};
\node [below] at (2,0) {\scriptsize$1'$};
\node [below] at (4,0) {\scriptsize$2'$};
\node [below] at (6,0) {\scriptsize$0'$};
\node [left] at (0,0) {\scriptsize$0''$};
\node [left] at (0,2) {\scriptsize$1''$};
\node [left] at (0,4) {\scriptsize$2''$};
\node [left] at (0,6) {\scriptsize$0''$};
\node [above] at (0,6) {\scriptsize$0'$};
\node [above] at (2,6) {\scriptsize$1'$};
\node [above] at (4,6) {\scriptsize$2'$};
\node [above] at (6,6) {\scriptsize$0'$};
\node [right] at (6,0) {\scriptsize$0''$};
\node [right] at (6,2) {\scriptsize$1''$};
\node [right] at (6,4) {\scriptsize$2''$};
\node [right] at (6,6) {\scriptsize$0''$};
\node [below] at (0.8,5) {$\bf 1$};
\node [above] at (1.4,4.8) {$\bf 1$};
\node [below] at (2.8,5) {$\bf 0$};
\node [above] at (3.4,4.8) {$\bf 0$};
\node [below] at (5.2,5) {$\bf 1$};
\node [above] at (4.65,4.8) {$\bf 0$};
\node [below] at (1.3,3) {$\bf 0$};
\node [above] at (.8,2.9) {$\bf 0$};
\node [below] at (3.3,3) {$\bf 1$};
\node [above] at (2.8,2.9) {$\bf 0$};
\node [below] at (4.7,3) {$\bf 1$};
\node [above] at (5.3,2.9) {$\bf 1$};
\node [below] at (1.35,1) {$\bf 1$};
\node [above] at (.6,0.9) {$\bf 0$};
\node [below] at (3.3,1) {$\bf 0$};
\node [above] at (2.7,0.9) {$\bf 1$};
\node [below] at (4.7,1) {$\bf 0$};
\node [above] at (5.3,0.9) {$\bf 0$};
\end{tikzpicture}$$
\caption{An optimized covering of the ordered product $\partial\Delta^2\times \partial\Delta^2$.}
\label{s1xs1}
\end{figure}
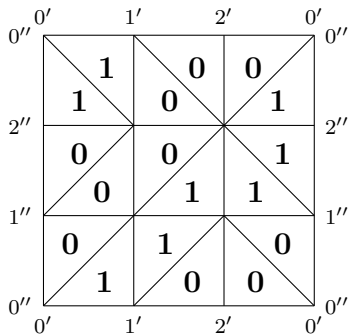

\begin{figure}[H]
\centering
	\subfloat[Output of \emph{Covering}.]{\begin{tikzpicture}
		\tikzset{dot/.style={circle,fill=#1,inner sep=0,minimum size=2.5pt}}
		\definecolor{userFillColour}{rgb}{0.8,0.8,0.8}
		\coordinate (v0) at (0,0);
		\coordinate (v1) at (1,0);
		\coordinate (v2) at (2,0);
		\coordinate (v3) at (3,0);
		\coordinate (v4) at (0,1);
		\coordinate (v5) at (1,1);
		\coordinate (v6) at (2,1);
		\coordinate (v7) at (3,1);
		\coordinate (v8) at (0,2);
		\coordinate (v9) at (1,2);
		\coordinate (v10) at (2,2);
		\coordinate (v11) at (3,2);
		\coordinate (v12) at (0,3);
		\coordinate (v13) at (1,3);
		\coordinate (v14) at (2,3);
		\coordinate (v15) at (3,3);
		%\fill[userFillColour] (v12) rectangle (v3);
		
		\draw [black,fill=white] (v5) -- (v6)-- (v10)--cycle; %(1,1),(2,1),(2,2)
		\draw [black,fill=white] (v15) -- (v14)-- (v10)--cycle; %(0, 0), (2, 0), (2, 2)
		\draw [black,fill=white] (v3) -- (v7)-- (v6)--cycle; %(0, 0), (0, 1), (2, 1)
		\draw [black,fill=white] (v4) -- (v8)-- (v9)--cycle; %(0, 1), (0, 2), (1, 2)
		\draw [black,fill=white] (v12) -- (v13)-- (v9)--cycle; %(0, 0), (1, 0), (1, 2)
		\draw [black,fill=white] (v7) -- (v6)-- (v10)--cycle; %(0, 1), (2, 1), (2, 2)
		\draw [black,fill=white] (v7) -- (v11)-- (v10)--cycle; %(0, 1), (0, 2), (2, 2)
		\draw [black,fill=white] (v15) -- (v11)-- (v10)--cycle; %(0, 0), (0, 2), (2, 2)
		\draw [black,fill=white] (v0) -- (v1)-- (v5)--cycle; %(0, 0), (1, 0), (1, 1)
		\draw [black,fill=white] (v1) -- (v2)-- (v6)--cycle; %(1, 0), (2, 0), (2, 1)
		\draw [black,fill=white] (v0) -- (v4)-- (v5)--cycle; %(0, 0), (0, 1), (1, 1)
		\draw [black,fill=white] (v12) -- (v8)-- (v9)--cycle; %(0, 0), (0, 2), (1, 2)
		\draw [black,fill=white] (v1) -- (v5)-- (v6)--cycle; %(1, 0), (1, 1), (2, 1)
		\draw [black,fill=white] (v3) -- (v2)-- (v6)--cycle; %(0, 0), (2, 0), (2, 1)
		\draw [black,fill=white] (v4) -- (v5)-- (v9)--cycle; %(0, 1), (1, 1), (1, 2)
		\draw [black,fill=white] (v13) -- (v9)-- (v10)--cycle; %(1, 0), (1, 2), (2, 2)
		\draw [black,fill=white] (v5) -- (v9)-- (v10)--cycle; %(1, 1), (1, 2), (2, 2)
		\draw [black,fill=white] (v13) -- (v14)-- (v10)--cycle; %(1, 0), (2, 0), (2, 2)
		
		\draw [black,fill=lightgray]  (v0) -- (v1)-- (v5) --cycle;
		\draw [black,fill=lightgray]  (v15) -- (v11)-- (v10) --cycle;
		\draw [black,fill=lightgray]  (v12) -- (v13)-- (v9) --cycle;
		\draw [black,fill=lightgray]  (v1) -- (v5)-- (v6) --cycle;
		\draw [black,fill=lightgray]  (v12) -- (v8)-- (v9) --cycle;
		\draw [black,fill=lightgray]  (v5) -- (v6)-- (v10) --cycle;
		\draw [black,fill=lightgray]  (v15) -- (v14)-- (v10) --cycle;
		\draw [black,fill=lightgray]  (v0) -- (v4)-- (v5) --cycle;
		\draw [black,fill=gray]  (v3) -- (v2)-- (v6) --cycle;
		\draw [black,fill=gray]  (v4) -- (v5)-- (v9) --cycle;
		\draw [black,fill=gray]  (v3) -- (v7)-- (v6) --cycle;
		\draw [black,fill=gray]  (v7) -- (v11)-- (v10) --cycle;
		\draw [black,fill=gray]  (v4) -- (v8)-- (v9) --cycle;
		\draw [black,fill=gray]  (v7) -- (v6)-- (v10) --cycle;
		\draw [black,fill=gray]  (v1) -- (v2)-- (v6) --cycle;
		\draw [black,fill=darkgray]  (v13) -- (v9)-- (v10) --cycle;
		\draw [black,fill=darkgray]  (v13) -- (v14)-- (v10) --cycle;
		\draw [black,fill=darkgray]  (v5) -- (v9)-- (v10) --cycle;
		
		\node [dot=white,draw=black] at (v0) {}; %inferior izquierdo
		\node [dot=white,draw=black] at (v1) {};%abajo
		\node [dot=white,draw=black] at (v2) {};%abajo
		\node [dot=white,draw=black] at (v3) {};%inferior derecho
		\node [dot=white,draw=black] at (v4) {}; %izquierdo
		\node [dot=white,draw=black] at (v5) {};
		\node [dot=white,draw=black] at (v6) {};
		\node [dot=white,draw=black] at (v7) {};%derecho
		\node [dot=white,draw=black] at (v8) {};%izquierdo
		\node [dot=white,draw=black] at (v9) {};
		\node [dot=white,draw=black] at (v10) {};
		\node [dot=white,draw=black] at (v11) {};%derecho
		\node [dot=white,draw=black] at (v12) {};%superior izquierdo
		\node [dot=white,draw=black] at (v13) {};%superior
		\node [dot=white,draw=black] at (v14) {};%superior
		\node [dot=white,draw=black] at (v15) {};%superior derecho
		
		\node [] at (v0) {\tiny $(0, 0)$}; %inferior izquierdo
		\node [] at (v1) {\tiny $(1,0)$};%abajo
		\node [] at (v2) {\tiny $(2,0)$};%abajo
		\node [] at (v3) {\tiny $(0,0)$};%inferior derecho
		\node [] at (v4) {\tiny $(0,1)$}; %izquierdo
		\node [] at (v5) {\tiny $(1,1)$};
		\node [] at (v6) {\tiny $(2,1)$};
		\node [] at (v7) {\tiny $(0,1)$};%derecho
		\node [] at (v8) {\tiny $(0,2)$};%izquierdo
		\node [] at (v9) {\tiny $(1,2)$};
		\node [] at (v10) {\tiny $(2,2)$};
		\node [] at (v11) {\tiny $(0,2)$};%derecho
		\node [] at (v12) {\tiny $(0,0)$};%superior izquierdo
		\node [] at (v13) {\tiny $(1,0)$};%superior
		\node [] at (v14) {\tiny $(2,0)$};%superior
		\node [] at (v15) {\tiny $(0,0)$};%superior derecho										
		\end{tikzpicture}}~\subfloat[Third iteration.]{\begin{tikzpicture}
		\tikzset{dot/.style={circle,fill=#1,inner sep=0,minimum size=2.5pt}}
		\definecolor{userFillColour}{rgb}{0.8,0.8,0.8}
		\coordinate (v0) at (0,0);
		\coordinate (v1) at (1,0);
		\coordinate (v2) at (2,0);
		\coordinate (v3) at (3,0);
		\coordinate (v4) at (0,1);
		\coordinate (v5) at (1,1);
		\coordinate (v6) at (2,1);
		\coordinate (v7) at (3,1);
		\coordinate (v8) at (0,2);
		\coordinate (v9) at (1,2);
		\coordinate (v10) at (2,2);
		\coordinate (v11) at (3,2);
		\coordinate (v12) at (0,3);
		\coordinate (v13) at (1,3);
		\coordinate (v14) at (2,3);
		\coordinate (v15) at (3,3);
		%\fill[userFillColour] (v12) rectangle (v3);							
		
		\draw [black,fill=white] (v5) -- (v6)-- (v10)--cycle; %(1,1),(2,1),(2,2)
		\draw [black,fill=white] (v15) -- (v14)-- (v10)--cycle; %(0, 0), (2, 0), (2, 2)
		\draw [black,fill=white] (v3) -- (v7)-- (v6)--cycle; %(0, 0), (0, 1), (2, 1)
		\draw [black,fill=white] (v4) -- (v8)-- (v9)--cycle; %(0, 1), (0, 2), (1, 2)
		\draw [black,fill=white] (v12) -- (v13)-- (v9)--cycle; %(0, 0), (1, 0), (1, 2)
		\draw [black,fill=white] (v7) -- (v6)-- (v10)--cycle; %(0, 1), (2, 1), (2, 2)
		\draw [black,fill=white] (v7) -- (v11)-- (v10)--cycle; %(0, 1), (0, 2), (2, 2)
		\draw [black,fill=white] (v15) -- (v11)-- (v10)--cycle; %(0, 0), (0, 2), (2, 2)
		\draw [black,fill=white] (v0) -- (v1)-- (v5)--cycle; %(0, 0), (1, 0), (1, 1)
		\draw [black,fill=white] (v1) -- (v2)-- (v6)--cycle; %(1, 0), (2, 0), (2, 1)
		\draw [black,fill=white] (v0) -- (v4)-- (v5)--cycle; %(0, 0), (0, 1), (1, 1)
		\draw [black,fill=white] (v12) -- (v8)-- (v9)--cycle; %(0, 0), (0, 2), (1, 2)
		\draw [black,fill=white] (v1) -- (v5)-- (v6)--cycle; %(1, 0), (1, 1), (2, 1)
		\draw [black,fill=white] (v3) -- (v2)-- (v6)--cycle; %(0, 0), (2, 0), (2, 1)
		\draw [black,fill=white] (v4) -- (v5)-- (v9)--cycle; %(0, 1), (1, 1), (1, 2)
		\draw [black,fill=white] (v13) -- (v9)-- (v10)--cycle; %(1, 0), (1, 2), (2, 2)
		\draw [black,fill=white] (v5) -- (v9)-- (v10)--cycle; %(1, 1), (1, 2), (2, 2)
		\draw [black,fill=white] (v13) -- (v14)-- (v10)--cycle; %(1, 0), (2, 0), (2, 2)
		
		\draw [black,fill=lightgray]  (v0) -- (v1)-- (v5) --cycle;
		\draw [black,fill=lightgray]  (v15) -- (v11)-- (v10) --cycle;
		\draw [black,fill=lightgray]  (v12) -- (v13)-- (v9) --cycle;
		\draw [black,fill=lightgray]  (v1) -- (v5)-- (v6) --cycle;
		\draw [black,fill=lightgray]  (v12) -- (v8)-- (v9) --cycle;
		\draw [black,fill=lightgray]  (v5) -- (v6)-- (v10) --cycle;
		\draw [black,fill=lightgray]  (v15) -- (v14)-- (v10) --cycle;
		\draw [black,fill=lightgray]  (v0) -- (v4)-- (v5) --cycle;
		\draw [black,fill=gray]  (v13) -- (v9)-- (v10) --cycle;
		\draw [black,fill=gray]  (v13) -- (v14)-- (v10) --cycle;
		\draw [black,fill=gray]  (v1) -- (v2)-- (v6) --cycle;
		\draw [black,fill=gray]  (v4) -- (v5)-- (v9) --cycle;
		\draw [black,fill=gray]  (v5) -- (v9)-- (v10) --cycle;
		\draw [black,fill=gray]  (v3) -- (v7)-- (v6) --cycle;
		\draw [black,fill=gray]  (v4) -- (v8)-- (v9) --cycle;
		\draw [black,fill=gray]  (v3) -- (v2)-- (v6) --cycle;
		\draw [black,fill=darkgray]  (v7) -- (v6)-- (v10) --cycle;
		\draw [black,fill=darkgray]  (v7) -- (v11)-- (v10) --cycle;
		
		\node [dot=white,draw=black] at (v0) {}; %inferior izquierdo
		\node [dot=white,draw=black] at (v1) {};%abajo
		\node [dot=white,draw=black] at (v2) {};%abajo
		\node [dot=white,draw=black] at (v3) {};%inferior derecho
		\node [dot=white,draw=black] at (v4) {}; %izquierdo
		\node [dot=white,draw=black] at (v5) {};
		\node [dot=white,draw=black] at (v6) {};
		\node [dot=white,draw=black] at (v7) {};%derecho
		\node [dot=white,draw=black] at (v8) {};%izquierdo
		\node [dot=white,draw=black] at (v9) {};
		\node [dot=white,draw=black] at (v10) {};
		\node [dot=white,draw=black] at (v11) {};%derecho
		\node [dot=white,draw=black] at (v12) {};%superior izquierdo
		\node [dot=white,draw=black] at (v13) {};%superior
		\node [dot=white,draw=black] at (v14) {};%superior
		\node [dot=white,draw=black] at (v15) {};%superior derecho
		
		\node [] at (v0) {\tiny $(0, 0)$}; %inferior izquierdo
		\node [] at (v1) {\tiny $(1,0)$};%abajo
		\node [] at (v2) {\tiny $(2,0)$};%abajo
		\node [] at (v3) {\tiny $(0,0)$};%inferior derecho
		\node [] at (v4) {\tiny $(0,1)$}; %izquierdo
		\node [] at (v5) {\tiny $(1,1)$};
		\node [] at (v6) {\tiny $(2,1)$};
		\node [] at (v7) {\tiny $(0,1)$};%derecho
		\node [] at (v8) {\tiny $(0,2)$};%izquierdo
		\node [] at (v9) {\tiny $(1,2)$};
		\node [] at (v10) {\tiny $(2,2)$};
		\node [] at (v11) {\tiny $(0,2)$};%derecho
		\node [] at (v12) {\tiny $(0,0)$};%superior izquierdo
		\node [] at (v13) {\tiny $(1,0)$};%superior
		\node [] at (v14) {\tiny $(2,0)$};%superior
		\node [] at (v15) {\tiny $(0,0)$};%superior derecho						
		\end{tikzpicture}
		
	}\\
	\subfloat[Sixth iteration.]{		
		\begin{tikzpicture}
		\tikzset{dot/.style={circle,fill=#1,inner sep=0,minimum size=2.5pt}}
		\definecolor{userFillColour}{rgb}{0.8,0.8,0.8}
		\coordinate (v0) at (0,0);
		\coordinate (v1) at (1,0);
		\coordinate (v2) at (2,0);
		\coordinate (v3) at (3,0);
		\coordinate (v4) at (0,1);
		\coordinate (v5) at (1,1);
		\coordinate (v6) at (2,1);
		\coordinate (v7) at (3,1);
		\coordinate (v8) at (0,2);
		\coordinate (v9) at (1,2);
		\coordinate (v10) at (2,2);
		\coordinate (v11) at (3,2);
		\coordinate (v12) at (0,3);
		\coordinate (v13) at (1,3);
		\coordinate (v14) at (2,3);
		\coordinate (v15) at (3,3);
		%\fill[userFillColour] (v12) rectangle (v3);								
		
		\draw [black,fill=white] (v5) -- (v6)-- (v10)--cycle; %(1,1),(2,1),(2,2)
		\draw [black,fill=white] (v15) -- (v14)-- (v10)--cycle; %(0, 0), (2, 0), (2, 2)
		\draw [black,fill=white] (v3) -- (v7)-- (v6)--cycle; %(0, 0), (0, 1), (2, 1)
		\draw [black,fill=white] (v4) -- (v8)-- (v9)--cycle; %(0, 1), (0, 2), (1, 2)
		\draw [black,fill=white] (v12) -- (v13)-- (v9)--cycle; %(0, 0), (1, 0), (1, 2)
		\draw [black,fill=white] (v7) -- (v6)-- (v10)--cycle; %(0, 1), (2, 1), (2, 2)
		\draw [black,fill=white] (v7) -- (v11)-- (v10)--cycle; %(0, 1), (0, 2), (2, 2)
		\draw [black,fill=white] (v15) -- (v11)-- (v10)--cycle; %(0, 0), (0, 2), (2, 2)
		\draw [black,fill=white] (v0) -- (v1)-- (v5)--cycle; %(0, 0), (1, 0), (1, 1)
		\draw [black,fill=white] (v1) -- (v2)-- (v6)--cycle; %(1, 0), (2, 0), (2, 1)
		\draw [black,fill=white] (v0) -- (v4)-- (v5)--cycle; %(0, 0), (0, 1), (1, 1)
		\draw [black,fill=white] (v12) -- (v8)-- (v9)--cycle; %(0, 0), (0, 2), (1, 2)
		\draw [black,fill=white] (v1) -- (v5)-- (v6)--cycle; %(1, 0), (1, 1), (2, 1)
		\draw [black,fill=white] (v3) -- (v2)-- (v6)--cycle; %(0, 0), (2, 0), (2, 1)
		\draw [black,fill=white] (v4) -- (v5)-- (v9)--cycle; %(0, 1), (1, 1), (1, 2)
		\draw [black,fill=white] (v13) -- (v9)-- (v10)--cycle; %(1, 0), (1, 2), (2, 2)
		\draw [black,fill=white] (v5) -- (v9)-- (v10)--cycle; %(1, 1), (1, 2), (2, 2)
		\draw [black,fill=white] (v13) -- (v14)-- (v10)--cycle; %(1, 0), (2, 0), (2, 2)
		
		\draw [black,fill=lightgray]  (v7) -- (v6)-- (v10) --cycle;
		\draw [black,fill=lightgray]  (v5) -- (v6)-- (v10) --cycle;
		\draw [black,fill=lightgray]  (v15) -- (v11)-- (v10) --cycle;
		\draw [black,fill=lightgray]  (v12) -- (v13)-- (v9) --cycle;
		\draw [black,fill=lightgray]  (v1) -- (v5)-- (v6) --cycle;
		\draw [black,fill=lightgray]  (v12) -- (v8)-- (v9) --cycle;
		\draw [black,fill=lightgray]  (v0) -- (v1)-- (v5) --cycle;
		\draw [black,fill=lightgray]  (v7) -- (v11)-- (v10) --cycle;
		\draw [black,fill=lightgray]  (v15) -- (v14)-- (v10) --cycle;
		\draw [black,fill=gray]  (v3) -- (v2)-- (v6) --cycle;
		\draw [black,fill=gray]  (v4) -- (v8)-- (v9) --cycle;
		\draw [black,fill=gray]  (v3) -- (v7)-- (v6) --cycle;
		\draw [black,fill=gray]  (v4) -- (v5)-- (v9) --cycle;
		\draw [black,fill=gray]  (v13) -- (v9)-- (v10) --cycle;
		\draw [black,fill=gray]  (v5) -- (v9)-- (v10) --cycle;
		\draw [black,fill=gray]  (v13) -- (v14)-- (v10) --cycle;
		\draw [black,fill=gray]  (v1) -- (v2)-- (v6) --cycle;
		\draw [black,fill=darkgray]  (v0) -- (v4)-- (v5) --cycle;
		
		\node [dot=white,draw=black] at (v0) {}; %inferior izquierdo
		\node [dot=white,draw=black] at (v1) {};%abajo
		\node [dot=white,draw=black] at (v2) {};%abajo
		\node [dot=white,draw=black] at (v3) {};%inferior derecho
		\node [dot=white,draw=black] at (v4) {}; %izquierdo
		\node [dot=white,draw=black] at (v5) {};
		\node [dot=white,draw=black] at (v6) {};
		\node [dot=white,draw=black] at (v7) {};%derecho
		\node [dot=white,draw=black] at (v8) {};%izquierdo
		\node [dot=white,draw=black] at (v9) {};
		\node [dot=white,draw=black] at (v10) {};
		\node [dot=white,draw=black] at (v11) {};%derecho
		\node [dot=white,draw=black] at (v12) {};%superior izquierdo
		\node [dot=white,draw=black] at (v13) {};%superior
		\node [dot=white,draw=black] at (v14) {};%superior
		\node [dot=white,draw=black] at (v15) {};%superior derecho
		
		\node [] at (v0) {\tiny $(0, 0)$}; %inferior izquierdo
		\node [] at (v1) {\tiny $(1,0)$};%abajo
		\node [] at (v2) {\tiny $(2,0)$};%abajo
		\node [] at (v3) {\tiny $(0,0)$};%inferior derecho
		\node [] at (v4) {\tiny $(0,1)$}; %izquierdo
		\node [] at (v5) {\tiny $(1,1)$};
		\node [] at (v6) {\tiny $(2,1)$};
		\node [] at (v7) {\tiny $(0,1)$};%derecho
		\node [] at (v8) {\tiny $(0,2)$};%izquierdo
		\node [] at (v9) {\tiny $(1,2)$};
		\node [] at (v10) {\tiny $(2,2)$};
		\node [] at (v11) {\tiny $(0,2)$};%derecho
		\node [] at (v12) {\tiny $(0,0)$};%superior izquierdo
		\node [] at (v13) {\tiny $(1,0)$};%superior
		\node [] at (v14) {\tiny $(2,0)$};%superior
		\node [] at (v15) {\tiny $(0,0)$};%superior derecho						
		
		\end{tikzpicture}}~\subfloat[Seventh iteration.]{\begin{tikzpicture}
		\tikzset{dot/.style={circle,fill=#1,inner sep=0,minimum size=2.5pt}}
		\definecolor{userFillColour}{rgb}{0.8,0.8,0.8}
		\coordinate (v0) at (0,0);
		\coordinate (v1) at (1,0);
		\coordinate (v2) at (2,0);
		\coordinate (v3) at (3,0);
		\coordinate (v4) at (0,1);
		\coordinate (v5) at (1,1);
		\coordinate (v6) at (2,1);
		\coordinate (v7) at (3,1);
		\coordinate (v8) at (0,2);
		\coordinate (v9) at (1,2);
		\coordinate (v10) at (2,2);
		\coordinate (v11) at (3,2);
		\coordinate (v12) at (0,3);
		\coordinate (v13) at (1,3);
		\coordinate (v14) at (2,3);
		\coordinate (v15) at (3,3);
		%\fill[userFillColour] (v12) rectangle (v3);								
		
		\draw [black,fill=white] (v5) -- (v6)-- (v10)--cycle; %(1,1),(2,1),(2,2)
		\draw [black,fill=white] (v15) -- (v14)-- (v10)--cycle; %(0, 0), (2, 0), (2, 2)
		\draw [black,fill=white] (v3) -- (v7)-- (v6)--cycle; %(0, 0), (0, 1), (2, 1)
		\draw [black,fill=white] (v4) -- (v8)-- (v9)--cycle; %(0, 1), (0, 2), (1, 2)
		\draw [black,fill=white] (v12) -- (v13)-- (v9)--cycle; %(0, 0), (1, 0), (1, 2)
		\draw [black,fill=white] (v7) -- (v6)-- (v10)--cycle; %(0, 1), (2, 1), (2, 2)
		\draw [black,fill=white] (v7) -- (v11)-- (v10)--cycle; %(0, 1), (0, 2), (2, 2)
		\draw [black,fill=white] (v15) -- (v11)-- (v10)--cycle; %(0, 0), (0, 2), (2, 2)
		\draw [black,fill=white] (v0) -- (v1)-- (v5)--cycle; %(0, 0), (1, 0), (1, 1)
		\draw [black,fill=white] (v1) -- (v2)-- (v6)--cycle; %(1, 0), (2, 0), (2, 1)
		\draw [black,fill=white] (v0) -- (v4)-- (v5)--cycle; %(0, 0), (0, 1), (1, 1)
		\draw [black,fill=white] (v12) -- (v8)-- (v9)--cycle; %(0, 0), (0, 2), (1, 2)
		\draw [black,fill=white] (v1) -- (v5)-- (v6)--cycle; %(1, 0), (1, 1), (2, 1)
		\draw [black,fill=white] (v3) -- (v2)-- (v6)--cycle; %(0, 0), (2, 0), (2, 1)
		\draw [black,fill=white] (v4) -- (v5)-- (v9)--cycle; %(0, 1), (1, 1), (1, 2)
		\draw [black,fill=white] (v13) -- (v9)-- (v10)--cycle; %(1, 0), (1, 2), (2, 2)
		\draw [black,fill=white] (v5) -- (v9)-- (v10)--cycle; %(1, 1), (1, 2), (2, 2)
		\draw [black,fill=white] (v13) -- (v14)-- (v10)--cycle; %(1, 0), (2, 0), (2, 2)
		
		\draw [black,fill=lightgray]  (v3) -- (v2)-- (v6) --cycle;
		\draw [black,fill=lightgray]  (v4) -- (v5)-- (v9) --cycle;
		\draw [black,fill=lightgray]  (v3) -- (v7)-- (v6) --cycle;
		\draw [black,fill=lightgray]  (v4) -- (v8)-- (v9) --cycle;
		\draw [black,fill=lightgray]  (v13) -- (v9)-- (v10) --cycle;
		\draw [black,fill=lightgray]  (v15) -- (v14)-- (v10) --cycle;
		\draw [black,fill=lightgray]  (v5) -- (v9)-- (v10) --cycle;
		\draw [black,fill=lightgray]  (v13) -- (v14)-- (v10) --cycle;
		\draw [black,fill=lightgray]  (v1) -- (v2)-- (v6) --cycle;
		\draw [black,fill=lightgray]  (v0) -- (v4)-- (v5) --cycle;
		\draw [black,fill=gray]  (v5) -- (v6)-- (v10) --cycle;
		\draw [black,fill=gray]  (v12) -- (v13)-- (v9) --cycle;
		\draw [black,fill=gray]  (v7) -- (v6)-- (v10) --cycle;
		\draw [black,fill=gray]  (v15) -- (v11)-- (v10) --cycle;
		\draw [black,fill=gray]  (v7) -- (v11)-- (v10) --cycle;
		\draw [black,fill=gray]  (v0) -- (v1)-- (v5) --cycle;
		\draw [black,fill=gray]  (v12) -- (v8)-- (v9) --cycle;
		\draw [black,fill=gray]  (v1) -- (v5)-- (v6) --cycle;
		
		\node [dot=white,draw=black] at (v0) {}; %inferior izquierdo
		\node [dot=white,draw=black] at (v1) {};%abajo
		\node [dot=white,draw=black] at (v2) {};%abajo
		\node [dot=white,draw=black] at (v3) {};%inferior derecho
		\node [dot=white,draw=black] at (v4) {}; %izquierdo
		\node [dot=white,draw=black] at (v5) {};
		\node [dot=white,draw=black] at (v6) {};
		\node [dot=white,draw=black] at (v7) {};%derecho
		\node [dot=white,draw=black] at (v8) {};%izquierdo
		\node [dot=white,draw=black] at (v9) {};
		\node [dot=white,draw=black] at (v10) {};
		\node [dot=white,draw=black] at (v11) {};%derecho
		\node [dot=white,draw=black] at (v12) {};%superior izquierdo
		\node [dot=white,draw=black] at (v13) {};%superior
		\node [dot=white,draw=black] at (v14) {};%superior
		\node [dot=white,draw=black] at (v15) {};%superior derecho
		
		\node [] at (v0) {\tiny $(0, 0)$}; %inferior izquierdo
		\node [] at (v1) {\tiny $(1,0)$};%abajo
		\node [] at (v2) {\tiny $(2,0)$};%abajo
		\node [] at (v3) {\tiny $(0,0)$};%inferior derecho
		\node [] at (v4) {\tiny $(0,1)$}; %izquierdo
		\node [] at (v5) {\tiny $(1,1)$};
		\node [] at (v6) {\tiny $(2,1)$};
		\node [] at (v7) {\tiny $(0,1)$};%derecho
		\node [] at (v8) {\tiny $(0,2)$};%izquierdo
		\node [] at (v9) {\tiny $(1,2)$};
		\node [] at (v10) {\tiny $(2,2)$};
		\node [] at (v11) {\tiny $(0,2)$};%derecho
		\node [] at (v12) {\tiny $(0,0)$};%superior izquierdo
		\node [] at (v13) {\tiny $(1,0)$};%superior
		\node [] at (v14) {\tiny $(2,0)$};%superior
		\node [] at (v15) {\tiny $(0,0)$};%superior derecho
		\end{tikzpicture}}
	\caption{Main steps in \emph{OptimizedCovering} in the case of $\partial\Delta^2$.\label{procesodecovering}}
\end{figure}
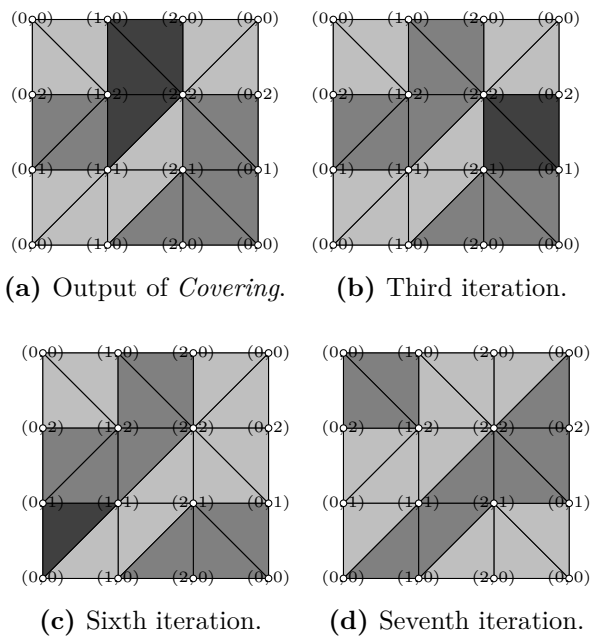

Relevant stages in the actual process performed by \emph{OptimizedCovering} that lead to the cover $\{J_0,J_1\}$ are illustrated in Figure~\ref{procesodecovering}. At each of the stages shown, the largest (smallest) subcomplex is highlighted in light (dark) gray. \emph{OptimizedCover} finishes when the dark gray subcomplex has been incorporated into two large subcomplexes. The corresponding contiguity chains (simplified by \emph{Reduce}) for the restricted projections $\pi_1|_{J_i},\pi_2|_{J_i}\colon J_i\to K$ are described in Table~\ref{tab:cuadro4} (for $i=0$) and Table~\ref{tab:cuadro3} (for $i=1$). In particular, the optimal $\SC_{\text{strict}}(\partial\Delta^2)=1$ is attained though contiguity chains of length at most $9$.

\begin{table}[H]
\centering
\begin{tabular}{|c|c|c|c|c|c|c|c|c|c|} \hline
	& $\left(0, 0\right)$ & $\left(0, 1\right)$ & $\left(0,
	2\right)$ & $\left(1, 0\right)$ & $\left(1, 1\right)$ &
	$\left(1, 2\right)$ & $\left(2, 0\right)$ & $\left(2, 1\right)$
	& $\left(2, 2\right)$ \\ \hline
	$\varphi_0=\pi_1|_{J_0}$	& $0$ & $0$ & $0$ & $1$ & $1$ & $1$ & $2$
	& $2$ & $2$ \\ \hline
	$\varphi_1$	& $0$ & $0$ & $0$ & $2$ & $1$ & $1$ & $2$
	& $2$ & $2$ \\ \hline
	$\varphi_2$	& $0$ & $0$ & $0$ & $2$ & $1$ & $1$ & $0$
	& $0$ & $2$ \\ \hline
	$\varphi_3$	& $0$ & $1$ & $1$ & $2$ & $1$ & $1$ & $0$
	& $0$ & $2$ \\ \hline
	$\varphi_4$	& $0$ & $1$ & $2$ & $2$ & $1$ & $2$ & $0$
	& $0$ & $2$ \\ \hline
	$\varphi_5$	& $0$ & $1$ & $2$ & $0$ & $1$ & $2$ & $0$
	& $0$ & $2$ \\ \hline
	$\varphi_6=\pi_2|_{J_0}$	& $0$ & $1$ & $2$ & $0$ & $1$ & $2$ & $0$
	& $1$ & $2$ \\ \hline
\end{tabular}
\caption{Contiguity chain between $\pi_1|_{J_0}$ and $\pi_2|_{J_0}$ for $J_0$ in Figure~\ref{s1xs1},
	\label{tab:cuadro4}}
\end{table}

\begin{table}[H]
\centering
\begin{tabular}{|c|c|c|c|c|c|c|c|c|} \hline
  & $\left(0, 0\right)$ & $\left(0, 1\right)$ & $\left(0, 2\right)$ & $\left(1, 0\right)$ & $\left(1, 1\right)$ & $\left(1, 2\right)$ & $\left(2, 1\right)$ & $\left(2, 2\right)$ \\ \hline
	$\varphi_0=\pi_1|_{J_1}$ & $0$ & $0$ & $0$ & $1$ & $1$ & $1$ & $2$ & $2$ \\ \hline
	$\varphi_1$ & $0$ & {$2$} & $0$ & $1$ & $1$ & $1$ & $2$ & $2$ \\ \hline
	$\varphi_2$ & $0$ & $2$ & $0$ & $1$ & $1$ & $1$ & {$1$} & $2$ \\ \hline
	$\varphi_3$ & $0$ & $2$ & $0$ & {$0$} & $1$ & $1$ & $1$ & $2$ \\ \hline
	$\varphi_4$ & $0$ & $2$ & $0$ & $0$ & $1$ & {$0$} & $1$ & $2$ \\ \hline
	$\varphi_5$ & $0$ & $2$ & {$2$} & $0$ & $1$ & $0$ & $1$ & $2$ \\ \hline
	$\varphi_6$ & $0$ & {$1$} & $2$ & $0$ & $1$ & $0$ & $1$ & $2$ \\ \hline
	$\varphi_7$ & $0$ & {$2$} & $2$ & $0$ & $1$ & $0$ & $1$ & $2$ \\ \hline
	$\varphi_8$ & $0$ & {$1$} & $2$ & $0$ & $1$ & $0$ & $1$ & $2$ \\ \hline
	$\varphi_9=\pi_2|_{J_1}$ & $0$ & $1$           & $2$ & $0$ & $1$ & {$2$} & $1$ & $2$ \\ \hline
\end{tabular}
\caption{Contiguity chain between $\pi_1|_{J_1}$ and $\pi_2|_{J_1}$ for $J_1$ in Figure~\ref{s1xs1}.
\label{tab:cuadro3}}
\end{table}

\begin{remark}{\em
Working instead with $\bsd(K\times K)$, \emph{OptimizedCovering} generates a system of piecewise linear motion planners using parameters $M=20000$ and $r=0.1$ for \emph{LocalSearch}. A detailed report of the resulting pair of piecewise linear domains and corresponding contiguity chains can be found in the Master's thesis of the first author. Despite the computer effort needed for such calculations, the output only gives that the optimal~(\ref{goalinicial}) can be attained though contiguity chains of length at most~26.
}\end{remark}

\subsection{A wedge of two circles}
Consider the 1-dimensional complex $K$ depicted in the following figure: 

$$
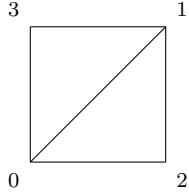
\begin{tikzpicture}[x=.6cm,y=.6cm]
\draw(0,0)--(3,0); \draw(0,0)--(0,3);
\draw(0,3)--(3,3); \draw(3,0)--(3,3);
\draw(0,0)--(3,3);
\node [below left] at (0,0) {\scriptsize$0$};
\node [below right] at (3,0) {\scriptsize$2$};
\node [above left] at (0,3) {\scriptsize$3$};
\node [above right] at (3,3) {\scriptsize$1$};
\end{tikzpicture}$$
The geometric realization $\|K\|$ is homotopic to a wedge of two circles. As in the case of a single circle, \emph{OptimizedCovering} recovers the well known equality $\TC(\|K\|)=2$ without requiring subdivisions. The $25$-minutes calculation was accomplished using parameters $r=0.1$ and $M=5000$ for \emph{LocalSearch}. The three resulting optimal piecewise linear local domains $J_0, J_1,J_2$ are generated by the following list of facets where, for simplicity, a vertex $(i,j)$ of the ordered product $K\times K$ is labelled as $4i+j$:

\begin{itemize}
\item[$J_0$:] $\{5, 9, 10\},\{0, 3, 15\},\{0, 4, 7\},\{4, 8, 11\},\{0, 4, 6\},\{5, 13, 15\},\{0, 3, 7\},$ 

$\{0, 2, 10\},\{0, 2, 6\},\{0, 4, 5\},\{0, 12, 14\},\{1, 13, 15\},\{0, 12, 15\},
\{4, 5, 9\},$

$\{4, 8, 9\},\{4, 5, 13\},\{4, 7, 11\},\{1, 3, 15\},\{0, 2, 14\},\{1, 3, 7\}.$

\item[$J_1$:] $\{4, 7, 15\},\{0, 1, 9\},\{1, 2, 6\},\{0, 1, 13\},\{4, 12, 13\},\{5, 7, 15\},\{1, 9, 10\},$

$\{1, 13, 14\},\{0, 1, 5\},\{1, 5, 7\},\{4, 12, 15\},\{0, 8, 9\},\{1, 5, 6\},\{0, 3, 11\},$

$\{0, 12, 13\},\{0, 8, 11\},\{1, 2, 10\},\{1, 2, 14\}$.

\item[$J_2$:] $\{4, 12, 14\},\{0, 8, 10\},\{1, 9, 11\},\{1, 3, 11\},\{5, 13, 14\},\{4, 6, 10\},\{4, 6, 14\},$

$\{5, 6, 10\},\{5, 6, 14\},\{4, 8, 10\},\{5, 7, 11\},\{5, 9, 11\}$.
\end{itemize}

\begin{table}[H]
	\begin{centering}
		\subfloat{\centering{}{%
\begin{tabular}{|c|c|c|c|c|c|c|c|c|c|c|c|c|c|c|c|c|} \hline
	& $0$ & $1$ & $2$ & $3$ & $4$ & $5$ & $6$
	& $7$ & $8$ & $9$ & $10$ & $11$ & $12$ &
	$13$ & $14$ & $15$ \\ \hline
	$\varphi_{0}=\pi_1|_{J_0}$ & $0$ & $0$ & $0$ & $0$ & $1$ & $1$
	& $1$ & $1$ & $2$ & $2$ & $2$ & $2$ & $3$
	& $3$ & $3$ & $3$ \\ \hline
	$\varphi_{1}$ & $0$ & $0$ & $0$ & $0$ & $1$ & $1$
	& $1$ & $1$ & $1$ & $1$ & $2$ & $2$ & $0$
	& $3$ & $0$ & $3$ \\ \hline
	$\varphi_{2}$ & $0$ & $0$ & $0$ & $0$ & $1$ & $1$
	& $0$ & $1$ & $1$ & $1$ & $2$ & $1$ & $0$
	& $3$ & $2$ & $3$ \\ \hline
	$\varphi_{3}$ & $0$ & $0$ & $2$ & $0$ & $1$ & $1$
	& $0$ & $0$ & $0$ & $1$ & $2$ & $1$ & $0$
	& $3$ & $0$ & $3$ \\ \hline
	$\varphi_{4}$ & $0$ & $3$ & $2$ & $3$ & $1$ & $1$
	& $0$ & $0$ & $0$ & $1$ & $2$ & $0$ & $0$
	& $3$ & $2$ & $3$ \\ \hline
	$\varphi_{5}$ & $0$ & $3$ & $2$ & $3$ & $1$ & $1$
	& $0$ & $0$ & $0$ & $1$ & $2$ & $0$ & $0$
	& $1$ & $2$ & $3$ \\ \hline
	$\varphi_{6}$ & $0$ & $3$ & $2$ & $3$ & $0$ & $1$
	& $0$ & $0$ & $0$ & $1$ & $2$ & $0$ & $0$
	& $1$ & $2$ & $3$ \\ \hline
	$\varphi_{7}$ & $0$ & $3$ & $2$ & $3$ & $0$ & $1$
	& $0$ & $3$ & $0$ & $1$ & $2$ & $0$ & $0$
	& $1$ & $2$ & $3$ \\ \hline
	$\varphi_{8}=\pi_2|_{J_0}$ & $0$ & $1$ & $2$ & $3$ & $0$ & $1$
	& $2$ & $3$ & $0$ & $1$ & $2$ & $3$ & $0$
	& $1$ & $2$ & $3$ \\ \hline
\end{tabular}}}
		\par\end{centering}
	\caption{\label{hjjsls}Contiguity chain on $J_0$.}
\end{table}

The corresponding contiguity chains are described in Tables~\ref{hjjsls}--\ref{jjyststs}.

\begin{table}[H]
	\begin{centering}
		\subfloat{\centering{}{%
\begin{tabular}{|c|c|c|c|c|c|c|c|c|c|c|c|c|c|c|c|c|} \hline
	& $0$ & $1$ & $2$ & $3$ & $4$ & $5$ & $6$
	& $7$ & $8$ & $9$ & $10$ & $11$ & $12$ &
	$13$ & $14$ & $15$ \\ \hline
	$\varphi_{0}=\pi_1|_{J_1}$ & $0$ & $0$ & $0$ & $0$ & $1$ & $1$
	& $1$ & $1$ & $2$ & $2$ & $2$ & $2$ & $3$
	& $3$ & $3$ & $3$ \\ \hline
	$\varphi_{1}$ & $0$ & $0$ & $0$ & $0$ & $1$ & $1$
	& $1$ & $1$ & $0$ & $2$ & $2$ & $0$ & $3$
	& $3$ & $3$ & $3$ \\ \hline
	$\varphi_{2}$ & $0$ & $0$ & $0$ & $1$ & $3$ & $1$
	& $0$ & $1$ & $0$ & $2$ & $2$ & $1$ & $3$
	& $3$ & $3$ & $3$ \\ \hline
	$\varphi_{3}$ & $0$ & $0$ & $0$ & $0$ & $3$ & $1$
	& $0$ & $1$ & $0$ & $2$ & $2$ & $1$ & $0$
	& $0$ & $0$ & $3$ \\ \hline
	$\varphi_{4}$ & $0$ & $0$ & $2$ & $0$ & $3$ & $1$
	& $0$ & $1$ & $0$ & $0$ & $0$ & $0$ & $0$
	& $0$ & $2$ & $3$ \\ \hline
	$\varphi_{5}$ & $0$ & $0$ & $0$ & $3$ & $3$ & $1$
	& $0$ & $1$ & $0$ & $0$ & $0$ & $3$ & $0$
	& $0$ & $0$ & $3$ \\ \hline
	$\varphi_{6}$ & $0$ & $1$ & $0$ & $3$ & $3$ & $1$
	& $0$ & $1$ & $0$ & $1$ & $0$ & $3$ & $0$
	& $0$ & $0$ & $3$ \\ \hline
	$\varphi_{7}$ & $0$ & $1$ & $0$ & $3$ & $3$ & $1$
	& $0$ & $3$ & $0$ & $1$ & $0$ & $3$ & $0$
	& $0$ & $0$ & $3$ \\ \hline
	$\varphi_{8}$ & $0$ & $1$ & $0$ & $3$ & $0$ & $1$
	& $0$ & $3$ & $0$ & $1$ & $1$ & $3$ & $0$
	& $0$ & $0$ & $3$ \\ \hline
	$\varphi_{9}$ & $0$ & $1$ & $1$ & $3$ & $0$ & $1$
	& $0$ & $3$ & $0$ & $1$ & $1$ & $3$ & $0$
	& $1$ & $1$ & $3$ \\ \hline
	$\varphi_{10}$ & $0$ & $1$ & $1$ & $3$ & $0$ & $1$
	& $1$ & $3$ & $0$ & $1$ & $2$ & $3$ & $0$
	& $1$ & $2$ & $3$ \\ \hline
	$\varphi_{11}=\pi_2|_{J_1}$ & $0$ & $1$ & $2$ & $3$ & $0$ & $1$
	& $2$ & $3$ & $0$ & $1$ & $2$ & $3$ & $0$
	& $1$ & $2$ & $3$ \\ \hline
\end{tabular}}}
		\par\end{centering}
	\caption{Contiguity chain on $J_1$.}
\end{table}

\begin{table}[H]
	\begin{centering}
		\subfloat{\centering{}{%
\begin{tabular}{|c|c|c|c|c|c|c|c|c|c|c|c|c|c|c|} \hline
	& $0$ & $1$ & $3$ & $4$ & $5$ & $6$ & $7$
	& $8$ & $9$ & $10$ & $11$ & $12$ & $13$ &
	$14$ \\ \hline
	$\varphi_{0}=\pi_1|_{J_2}$ & $0$ & $0$ & $0$ & $1$ & $1$ & $1$
	& $1$ & $2$ & $2$ & $2$ & $2$ & $3$ & $3$
	& $3$ \\ \hline
	$\varphi_{1}$ & $0$ & $0$ & $0$ & $1$ & $1$ & $1$
	& $1$ & $2$ & $2$ & $2$ & $2$ & $3$ & $1$
	& $1$ \\ \hline
	$\varphi_{2}$ & $2$ & $0$ & $0$ & $1$ & $1$ & $2$
	& $1$ & $2$ & $2$ & $2$ & $2$ & $3$ & $1$
	& $1$ \\ \hline
	$\varphi_{3}$ & $1$ & $2$ & $0$ & $1$ & $1$ & $2$
	& $1$ & $2$ & $2$ & $2$ & $2$ & $1$ & $1$
	& $1$ \\ \hline
	$\varphi_{4}$ & $1$ & $2$ & $2$ & $1$ & $1$ & $2$
	& $2$ & $2$ & $1$ & $2$ & $2$ & $2$ & $1$
	& $2$ \\ \hline
	$\varphi_{5}$ & $1$ & $2$ & $1$ & $1$ & $1$ & $2$
	& $1$ & $1$ & $1$ & $2$ & $1$ & $2$ & $1$
	& $2$ \\ \hline
	$\varphi_{6}$ & $1$ & $1$ & $1$ & $1$ & $1$ & $2$
	& $3$ & $1$ & $1$ & $2$ & $1$ & $2$ & $1$
	& $1$ \\ \hline
	$\varphi_{7}$ & $2$ & $3$ & $3$ & $2$ & $1$ & $2$
	& $3$ & $2$ & $1$ & $2$ & $3$ & $1$ & $1$
	& $2$ \\ \hline
	$\varphi_{8}$ & $0$ & $1$ & $3$ & $2$ & $1$ & $1$
	& $3$ & $0$ & $1$ & $2$ & $3$ & $2$ & $1$
	& $2$ \\ \hline
	$\varphi_{9}$ & $0$ & $1$ & $3$ & $2$ & $1$ & $2$
	& $3$ & $0$ & $1$ & $2$ & $3$ & $0$ & $1$
	& $2$ \\ \hline
	$\varphi_{10}=\pi_2|_{J_2}$ & $0$ & $1$ & $3$ & $0$ & $1$ & $2$
	& $3$ & $0$ & $1$ & $2$ & $3$ & $0$ & $1$
	& $2$ \\ \hline
\end{tabular}}}
		\par\end{centering}
	\caption{\label{jjyststs}Contiguity chain on $J_2$.}
\end{table}

%\bibliographystyle{plain}
%\bibliography{bib}

\bigskip

\noindent
{\sc Escuela Superior de F\'isica y Matem\'aticas del Instituto Polit\'ecnico Nacional. Edificio 9, Unidad Profesional Adolfo L\'opez Mateos, 07300, Mexico City, Mexico.

\noindent {\tt arweyo@gmail.com, alaral@ipn.mx.}}

\medskip\noindent
{\sc Departamento de Matem\'aticas, Centro de Investigaci\'on y de Estudios Avanzados del Istituto Polit\'ecnico Nacional, Av.~IPN 2508, Zacatenco, M\'exico City 07000, M\'exico.

\noindent {\tt jesus@math.cinvestav.mx}}

\medskip\noindent
{\sc Department of Mathematics, Faculty of Engineering and Natural Sciences, Bursa Technical University, Bursa, Turkey.

\noindent {\tt ayse.borat@btu.edu.tr}}

\end{document}